\documentclass[final,leqno,onefignum,onetabnum]{siamltex1213}
\usepackage{amsmath}
\usepackage{setspace}
\title{Upper bound on the block transposition diameter of the symmetric group}
\author{Annachiara Korchmaros \thanks{Dipartimento di Matematica e Informatica, University of Perugia, via Vanvitelli 1 06123, Italy.
(\email{annachiara.korchmaros@dmi.unipg.it}).}}
\newtheorem{thm}{Theorem}[section]
\newtheorem{prop}[thm]{Proposition}
\newtheorem{lem}[thm]{Lemma}
\newtheorem{crt}[thm]{Criterion}

\newtheorem{example}[thm]{Example}
\newtheorem{rem}[thm]{Remark}
\newtheorem{cor}[thm]{Corollary}
\begin{document}
\maketitle
\slugger{mms}{xxxx}{xx}{x}{x--x}
\begin{abstract}
Given a generator set $S$ of the symmetric group ${\rm{Sym}}_n$, every permutation $\pi\in {\rm{Sym}_n}$ is a word (product of elements) of $S$. A positive integer $d(\pi)$ is associated with each $\pi\in{\rm{Sym}_n}$ taking the length of the shortest such word, and
the $S$-diameter $d(S)$ is the maximum value of $d(\pi)$ with $\pi$ ranging over ${\rm{Sym_n}}$. The distance $d(\pi,\nu)$ of two permutations $\pi,\nu$ defined by $d(\nu^{-1}\circ\pi)$ satisfies the axioms of a metric space. In this paper we consider the case where $S$ consists of all block transpositions of ${\rm{Sym_n}}$ and call $d(\pi)$ the block transposition distance of $\pi$. A strong motivation for the study of this special case comes from investigations of large-scale mutations of genome, where determining $d(\pi)$ is known as sorting the permutation $\pi$ by block transpositions. In the papers on this subject, toric equivalence classes often play a crucial role since $d(\pi)=d(\nu)$ when $\pi$ and $\nu$ are torically equivalent. A proof of this result can be found in the (unpublished) Hausen's Ph.D Dissertation thesis; see \cite{Ha}. Our main contribution is to obtain a bijective map on ${\rm{Sym}_n}$ from the toric equivalence that leaves the distances invariant. Using the properties of this map, we give an alternative proof of Hausen's result which actually fills a gap in the proof of the upper bound on $d(S)$ due to Eriksson and his coworkers; see \cite{EE}. We also revisit the proof of the key lemma \cite[Lemma 5,1]{EE}, giving more details and filling some gaps.
\end{abstract}
\pagestyle{myheadings}
\thispagestyle{plain}
\markboth{BLOCK TRANSPOSITION DISTANCE}{A.~KORCHMAROS}
\section{Introduction}\label{intro}
The general problem of determining the diameters of generator sets of ${\rm{Sym}}_n$ has been intensively investigated in
Combinatorial Group Theory and Enumerative Combinatorics. Its study has also been motivated and stimulated by practical applications, especially in Computational Biology, where the choice of $S$ depends on a practical need that may not have straightforward connection with pure Mathematics; see section \ref{biol}. In this paper we deal with ``\emph{sorting a permutation $\pi$ by block transpositions}''. This problem asks for the block transposition distance $d(\pi)$ with respect to the generator set $S$ of certain permutations (\ref{eq22ott12}) called block transpositions. Although the distribution of the block transposition distances has been computed for $3\leq n\leq 10$; see \cite{FL}, L. Bulteau and his coworkers \cite{BF} proved that the problem of sorting a permutation by block transpositions is NP-hard. Therefore, it is challenging to determine the block transposition diameter $d(n)$ for larger $n$.
{}From \cite{EE}, $d(n)$ is known for $n\leq 15$:
\[\arraycolsep=1.4pt\def\arraystretch{2.2}
d(n)=\left\{
\begin{array}{cc}
\left\lceil\dfrac{n+1}{2}\right\rceil, &  3\leq n\leq 12\quad n=14,\\
\dfrac{n+3}{2}, & n=13,15.
\end{array}
\right.
\]
For $n>15$, estimates on the diameter $d(n)$ give useful information and have been the subject of several papers in the last decade. The best known lower bound on $d(n)$ is
$$d(n)\geq \left\{
\begin{aligned}
\frac{n+2}{2},\,\, & n \mbox{ is even,}\\
\frac{n+3}{2},\,\, &  n \mbox{ is odd;}
\end{aligned}
\right.
$$
see \cite{CK}, which improves previous bounds obtained in \cite{EE}. Regarding upper bounds, the strongest one available in the literature is Eriksson's bound, stated in 2001, see \cite{EE}: For $n\geq 3$, $$d(n)\leq \left\lfloor\frac{2n-2}{3}\right\rfloor.$$ However, the proof of Eriksson's bound given in \cite{EE} is incomplete, since it implicitly relies on the invariance of $d(\pi)$ when $\pi$ ranges over a toric class.
It should be noticed that this invariance principle has been claimed explicitly in a paper appeared in a widespread journal only recently; see \cite{CK}, although Hausen had already mentioned it and sketched a proof in his unpublished Ph.D Dissertation thesis; see \cite{Ha}. Elias and Hartman were not aware of Hausen's work and quoted Eriksson's bound in a weaker form which is independent of the invariance principle, see \cite{EH}, Proposition \ref{propEE}. In this paper we show how from the toric equivalence relation can be obtained a bijective map on ${\rm{Sym}_n}$ that leaves the distances invariant; see Definition \ref{defOmega}. Using the properties of this map, we give an alternative proof for the above invariance principle which we state in Theorem \ref{torically} and Theorem \ref{principle}. We also revisit the proof of the key lemma in \cite{EE}; see Proposition \ref{propEE}, giving more technical details and filling some gaps.
\section{Connection with Biology}\label{biol}
It is known that DNA segments evolve by small and large mutations. However, homologous segments of DNA (segments deriving from a common ancestor) are rarely subject to large mutations, named \emph{rearrangements}, which are structural variations of a segment. Such mutations are the subject of the \emph{genome rearrangement} problems. For the seek of simplicity, we say genes for homologous markers of DNA (segments extractable from species which support the hypothesis that they belonged to the common ancestor of these species), chromosome for the set of genes, and \emph{genome} for the set of all chromosomes of a given specie. With some natural restrictions of genome structure, it is possible to represent genomes by permutations on $\{1,2,\ldots,n\},$ where the labels are genes; see \cite{FL}. \emph{Intrachromosomal translocations} are rearrangements so that a segment of DNA is moved to another part of the same chromosome.
In the relative mathematical model, a special generator set $S$ of ${\rm{Sym_n}}$ consisting of \emph{block transpositions} is defined, and the block transpositions act on permutations by switching two adjacent subsequences. Since rearrangements are relatively rare events, the study of genome rearrangements is based on the so-called parsimony criterion (scenarios minimizing the number of mutations are more likely to be closed to reality); see \cite{FL}. Furthermore, working out an evolutionary scenario between two species requires to solve the problem of transforming a permutation to another by a minimum number of rearrangements. By Corollary \ref{sorting}, when $S$ is the set of block transpositions, the genome rearrangement problem is equivalent to the problem of sorting a permutation by block transpositions. If $\pi$ and $\rho$ are two chromosomes on the same set of genes $\{1,\,2,\cdots, n\}$, we denote by $d(\pi,\rho)$ the minimum number of block transpositions needed to transform $\pi$ into $\rho$. By Corollary \ref{sorting}, $d(\pi,\rho)=d(\delta),$ where $\delta=\pi\circ\rho^{-1}$ and $d(\delta)$ stands for $d(\delta,id)$. For further reference see \cite{BP}.
\section{Notation, definitions, and earlier results}
\label{definitions}
In this paper, a permutation
$$\pi=\left(
\begin{array}{cccc}
 1 & 2 & \cdots & n \\
\pi_1 & \pi_2 & \cdots & \pi_n \\
\end{array}
\right)$$on $[n]=\{1,\,2,\cdots,\,n\}$ is denoted by $\pi=[\pi_1\pi_2\cdots \pi_n]$. In particular, the \emph{reverse permutation} is $w_0=[n\,n-1\cdots 1]$ and $id=[1\,2\cdots n]$ is the \emph{identity permutation}. According to the notation adopted in \cite{EE}, the product $\pi\circ\nu$ of the permutations $\pi=[\pi_1\,\pi_2\,\cdots \pi_n]$ and $\nu=[\nu_1\,\nu_2\,\cdots \nu_n]$ on $[n]$ is defined to be their composite function, that is, $(\pi\circ\nu)_i=\pi_{\nu_i}$ for every $i\in [n]$. For any three integers, named \emph{cut points}, $i,j,k$ with $0\leq i< j< k\leq n$, the \emph{block transposition} $\sigma(i,j,k)$ acts on a permutation $\pi$ on $[n]$ switching two adjacent subsequences of $\pi$, named \emph{blocks}, without altering the order of integers within each block. The permutation $\sigma(i,j,k)$ is formally defined as follows
\begin{equation}
\label{eq22ott12}
\sigma(i,j,k)=\left\{\begin{array}{ll}
[1\cdots i\,\, j+1\cdots k\,\, i+1\cdots j\,\, k+1 \cdots n], & 1\leq i,\,k< n,\\
{[j+1\cdots k\,\, 1\cdots j\,\, k+1 \cdots n]}, & i=0,\,k< n,\\
{[1\cdots i\,\,j+1\cdots n\,\, i+1\cdots j]}, & 1\leq i,\,k=n,\\
{[j+1\cdots n\,\, 1\cdots j]}, & i=0,\,k=n.
\end{array}
\right.
\end{equation}This shows that $\sigma(i,j,k)_{t+1}=\sigma(i,j,k)_{t} +1$ in the intervals:
\begin{equation}
\label{function}
[1,i],\quad[i+1, k-j+i],\quad[k-j+i+1,k],\quad[k+1, n],
\end{equation}where
\begin{equation}
\label{cuppoints}
\begin{array}{l}
\sigma(i,j,k)_i=i,\quad\sigma(i,j,k)_{i+1}=j+1,\quad\sigma(i,j,k)_{k-j+i}=k,\\
\sigma(i,j,k)_{k-j+i+1}=i+1,\quad\sigma(i,j,k)_{k}=j,\quad\sigma(i,j,k)_{k+1}=k+1.
\end{array}
\end{equation}
The action of $\sigma(i,j,k)$ on $\pi$ is defined as the product $$\pi\circ\sigma(i,j,k)=[\pi_1\cdots \pi_i\,\,\pi_{j+1}\cdots \pi_k\,\,\pi_{i+1}\cdots \pi_{j}\,\,\pi_{k+1}\cdots \pi_n].$$ Therefore, applying a block transposition on the right of $\pi$ consists in switching two adjacent subsequences of $\pi$ or, equivalently, moving forward a block of $\pi$; see \cite{B}. This may also be expressed by $$[\pi_1\cdots\pi_i|\pi_{i+1}\cdots\pi_j|\pi_{j+1}\cdots\pi_k|\pi_{k+1}\cdots\pi_n].$$
It is straightforward to check that $\sigma(i,j,k)^{-1}=\sigma(i,k-j+i,k)$. This shows that the set $S$ of block transpositions is inverse-closed.
In section \ref{intro}, we have claimed that $S$ is also a generator set of ${\rm{Sym_n}}$. It should be noticed that this follows from known results of Group Theory since for every $0\le i<k\le n$ $$\sigma(i,i+1,k)=(i+1,\cdots, k),$$ where $(i+1,\cdots, k)$ is a cycle of ${\rm{Sym_n}}$.
Since $S$ is a generator set, the following definition is meaningful.
\begin{definition}
Let $\pi$ a permutation on $[n]$. The \emph{block transposition distance of $\pi$} is $d(\pi)$ if $\pi$ is the product of $d(\pi)$ block transpositions, but it cannot be obtained as the product of less than $d(\pi)$ block transpositions.
\end{definition}

Let $\pi,\nu\in {\rm{Sym_n}}$. Since $S$ is a generator set, there exist $\sigma_1,\cdots,\sigma_k\in S$ such that
\begin{equation}\label{dis}
\nu=\pi\circ\sigma_1\circ\cdots\circ\sigma_k.
\end{equation}The minimum number $d(\pi,\nu)$ of block transpositions occurring in (\ref{dis}) is the \emph{block transposition distance of $\pi$ and $\nu$}. We may note that such number is at most $n-1$ since we can always obtain $\nu$ from $\pi$ by moving forward a block of $\nu$ of length one.
Therefore, the map $$d\colon{\rm{Sym}_n}\times{\rm{Sym}_n}\to \{0,\ldots,n-1\},$$ is a distance as having each of the following three properties:
\begin{romannum}
\item $d(\pi,\nu)\geq 0$ with equality if and only if $\pi=\nu$,
\item $d(\pi,\nu)=d(\nu,\pi)$,
\item $d(\pi,\nu)+d(\nu,\mu)\geq d(\pi,\mu)$,
\end{romannum}for any $\pi,\nu,\mu\in Sym_n$. We also observe that $d$ is left-invariant, that is, for any $\mu,\pi,\nu\in {\rm{Sym_n}}$, we have $d(\pi,\nu)=d(\mu\circ\pi,\mu\circ \nu)$.
\begin{prop}\label{distance}
The block transposition distance is a distance on ${\rm{Sym_n}}$ and left-invariant.
\end{prop}
\begin{rem}\label{link}
Note that $d(\pi)=d(id,\pi)=d(\pi,id)$ by Proposition \ref{distance}.
\end{rem}
Proposition \ref{distance} has the following consequences on $d$.
\begin{cor}\label{prop d}
For every $\pi,\mu\in {\rm{Sym_n}}$,
\begin{romannum}
\item $d(\pi)+d(\mu)\geq d(\pi,\mu)$,
\item $d(\mu\circ\pi,\mu)=d(\pi),$
\item $d(\pi)=d(\pi^{-1}).$
\end{romannum}
\end{cor}
\begin{cor}
\label{sorting}
Computing the block transposition distance of two permutations is equivalent to sorting a permutation by block transpositions.
\end{cor}
\begin{proof}
Since $d$ is left-invariant, for every $\mu,\nu\in{\rm{Sym_n}}$, $d(\pi,\nu)=d(\nu^{-1}\circ\pi,id)$  whence
$$\{d(\pi,\nu): \pi,\nu\in{\rm{Sym_n}}\}=\{d(\mu, id):\mu\in{\rm{Sym_n}}\}.$$The statement follows from Remark \ref{link}.\qquad\end{proof}

\section{Circular and Toric Permutation Classes}
\label{ctpc}
Before discussing the key lem-\\ma in \cite{EE} and stating the related contributions obtained in the present paper, it is convenient to exhibit some useful equivalence relations on permutations introduced by Eriksson and his coworkers; see \cite{EE}. For this purpose, they considered permutations on the set $[n]^0=\{0,1,\ldots,n\}$ and recovered the permutations $\pi=[\pi_1\,\cdots\pi_n]$ on $[n]$ in the form $[0\pi]$, where $[0\,\pi]$ stands for the permutation $[0\,\pi_1\,\cdots \pi_n]$ on $[n]^0$. In particular, block transpositions $\bar\sigma(i,j,k)$ on $[n]^0$ are defined as in (\ref{eq22ott12}), where $-1\leq i<j<k\leq n$, and $\bar{\sigma}(i,j,k)=[0\,\sigma(i,j,k)]$ holds if and only of $i\geq 0$.

They observed that the $n+1$ permutations arising from $\bar{\pi}\in {\rm{Sym}_n^0}$ under cyclic index shift form an equivalence class $\pi^\circ$ containing a unique permutation of the form $[0\,\pi]$. A formal definition of $\pi^\circ$ is given below.
\begin{definition}\label{circ}\rm{(see \rm{\cite{CK}})}Let $\pi$ be a permutation on $[n]$. The \emph{circular permutation class} $\pi^\circ$ is obtained from $\pi$ by inserting an extra element $0$ that is considered a predecessor of $\pi_1$ and a successor of $\pi_n$ and taking the equivalence class under cyclic index shift. So $\pi^\circ$ is circular in positions being represented by $[0\,\pi_1\cdots \pi_{n-1}\,\pi_n],\,[\pi_n\, 0\,\pi_1\\\cdots\pi_{n-1}],$ and so on. The \emph{linearization of a circular permutation $\pi^\circ$} is a permutation $\pi$ obtained by removing the element $0$ and letting its successor be the first element of $\pi$. It is customary to denote by $\pi^\circ$ any representative of the circular class of $\pi$ and $\equiv^\circ$ the equivalence relation.
\end{definition} A necessary and sufficient condition for a permutation $\bar\pi$ on $[n]^0$ to be in $\pi^\circ$ is the existence of an integer $r$ with $0\le r \le n$ such that
\begin{equation}\label{aug4}
\bar\pi_x=[0\,\pi]_{x+r}\quad \mbox{for}\quad 0\le x \le n,\end{equation}where the indices are taken mod $n+1$.
The second equivalence class is an expansion of the circular permutation class, as it also involves cyclic value shifts.
\begin{definition}\rm{(see \rm{\cite{CK}})}Let $\pi$ be a permutation on $[n]$, and let $m$ be an integer with $1\leq m\leq n$. The $m$-step cyclic value shift of the circular permutation $\pi^\circ$ is the circular permutation $m+\pi^\circ =[m\,m + \pi_1\cdots m+ \pi_n]$, where the integers are taken mod $n+1$. The \emph{toric class} $\pi_\circ^\circ$ is obtained from $\pi^\circ$ by taking the $m$-step cyclic value shifts of the circular permutations in $\pi^\circ$. So, $\pi_\circ^\circ$ is circular in values, as well as in positions.\end{definition} The toric class $\pi_\circ^\circ$ comprises at most $(n+1)^2$ permutations of ${\rm{Sym_n^0}},$
but it may consist of a smaller circular permutation class and even collapse to a unique permutation. The latter case occurs when $\pi$ is the reversal permutation or the identity permutation. The following example comes from \cite{la}.
\begin{example}\label{exla}
Let $n=7$, and let $\pi=[4\,1\, 6\, 2\, 5\, 7\, 3]$. Then $\pi^\circ=[0\,4\,1\, 6\, 2\, 5\, 7\, 3]$, and $\pi_\circ^\circ$ consists of the permutations below together with their circular classes.
$$\begin{array}{lll}
0 + \pi^\circ = [0\, 4\, 1\, 6\, 2\, 5\, 7\, 3],&1 + \pi^\circ = [1\, 5\, 2\, 7\, 3\, 6\, 0\, 4],&2 + \pi^\circ = [2\, 6\, 3\, 0\, 4\, 7\, 1\, 5]\\
3 + \pi^\circ = [3\, 7\, 4\, 1\, 5\, 0\, 2\, 6],&4 + \pi^\circ = [4\, 0\, 5\, 2\, 6\, 1\, 3\, 7],&5 + \pi^\circ = [5\, 1\, 6\, 3\, 7\, 2\, 4\, 0]\\
6 + \pi^\circ = [6\, 2\, 7\, 4\, 0\, 3\, 5\, 1],&7 + \pi^\circ = [7\, 3\, 0\, 5\, 1\, 4\, 6\, 2].\\
\end{array}$$
\end{example}A necessary and sufficient condition for two permutations $\bar{\pi},\bar{\pi}'\in\rm{Sym_n^0}$ to be in the same toric class is the existence of integers $r,s$ with $0\le r,s \le n$ such that $\pi'=\pi_{x+r}-\pi_s$ holds for every $1\le x \le n,$ where the indices are taken mod $n+1$. In particular, for $\bar{\pi}=[0\,\pi]$ and $\bar{\pi}'=[0\,\pi']$, this necessary and sufficient condition reads: there exists an integer $r$ with $0\le r \le n$ such that\begin{equation}\label{circ_pw}
\pi'=\pi_{x+r}-\pi_r\quad \mbox{for}\quad 1\le x \le n,
\end{equation}where the indices are taken mod $n+1$. This gives rise to the following definition already introduced in \cite[Definition 7.3]{la}, but not appearing explicitly in \cite{EE}.
\begin{definition}
Two permutations $\pi$ and $\pi'$ on $[n]$ are \emph{torically equivalent} if $[0\,\pi]$ and $[0\,\pi']$ are in the same toric class.
\end{definition}
In Example \ref{exla}, the torically equivalent permutations are
$$\begin{array}{llll}
{[4\, 1\, 6\, 2\, 5\, 7\, 3],}&[4\, 1\, 5\, 2\, 7\, 3\, 6],&[4\, 7\, 1\, 5\, 2\, 6\, 3],&[2\, 6\, 3\, 7\, 4\, 1\, 5],\\
{[5\, 2\, 6\, 1\, 3\, 7\, 4],}&[5\, 1\, 6\, 3\, 7\, 2\, 4],&[3\, 5\, 1\, 6\, 2\, 7\, 4],&[5\, 1\, 4\, 6\, 2\, 7\, 3].\\
\end{array}$$
Let $\alpha$ be $[1\,2\cdots n\,0]$, by Definition \ref{circ}, every circular permutation $\pi^\circ$ is the product of $[0\,\pi]$ by a power of $\alpha$, namely
\begin{equation}\label{alpha_powers}
\alpha^r_x\equiv x+r {\pmod{n+1}}\quad \mbox{for}\quad 0\leq x\leq n.
\end{equation}Take a permutations $\pi$ on $[n]$. From (\ref{aug4}) it follows that a necessary and sufficient condition for a permutation $\bar\pi$ on $[n]^0$ to be in $\pi^\circ$ is the existence of an integer $r$ with $0\le r \le n$ such that $\bar\pi=[0\,\pi]\circ\alpha^r$. Therefore, by (\ref{circ_pw}), a permutation $\pi'$ on $[n]$ is torically equivalent to $\pi$ if and only if
\begin{equation}\label{toric_zero}
[0\,\pi']=\alpha^{-\pi_r}\circ[0\,\pi]\circ\alpha^r\quad \mbox{for} \quad 0\le r \le n.
\end{equation}
Since $(\alpha^{-\pi_r}\circ[0\,\pi]\circ\alpha^r)_x=\pi_{x+r}-\pi_r$ for every $1\le x \le n$, the following map arises from the toric equivalence.
\begin{definition}\label{defOmega}
Let $\pi$ be a permutation on $[n]$, and let $r$ be an integer with $0\leq r\leq n$. The bijective map $\Omega_r$ on ${\rm{Sym}_n}$ defined by
$$\Omega_r(\pi)_x=\pi_{x+r}-\pi_r\quad \mbox{for} \quad 1\le x \le n$$
is the {\emph{toric map}} of ${\rm{Sym}_n}$ (with respect to the set $S$ of all block transpositions on $[n]$).
\end{definition}

Therefore, $\pi'$ is torically equivalent to $\pi$ if $\Omega_r(\pi)=\pi'$ for some integer $r$ with $0\leq r\leq n$. A major related result is the \emph{invariance principle} stated in the following two theorems.
\begin{thm}\label{torically}
If two permutations $\pi$ and $\pi'$ on $[n]$ are torically equivalent, then $d(\pi)=d(\pi')$.
\end{thm}
\begin{thm}\label{principle}
Let $\pi,\varphi,\nu,\mu$ be permutations on $[n]$ such that the toric map $\Omega_r$ takes $\pi$ to $\varpi$ and $\nu$ to $\mu$. Then $d(\pi,\nu)=d(\varpi,\mu)$.
\end{thm}

We give a proof of Theorem \ref{torically} and Theorem\ref{principle} in Section \ref{invariance}.

An important role in the investigations of bounds on $d(n)$ is played by the number of bonds of a permutation, where a \emph{bond} of a permutation $\pi\in {\rm{Sym_n}}$ consists of two consecutive integers $x,x+1$ in the sequence $0\,\pi_1 \cdots \pi_n\, n+1$. $[0\,\pi]$ has a bond if and only if $\pi$ has a bond. It is easily seen that any two permutations in the same toric class have the same number of bonds.
A bond of $\pi^\circ$ is any $2$-sequence $x\,\overline{x}$ of $\pi^\circ$, where $\overline{x}$ denotes the smallest non-negative integer congruent to $x+1$ mod $n+1$. As bonds are rotation-invariant, their number is an invariant of $\pi_\circ^\circ$. The main result on bonds is the following.
\begin{prop}{\rm{(see \cite[Lemma 5.1]{EE})}}\label{wrong}
Let $\pi$ be any permutation on $[n]$ other than the reverse permutation. Then there are block transpositions $\sigma$ and $\tau$ such that either $\pi\circ\sigma\circ\tau,$ or $\sigma\circ\pi\circ\tau,$ or $\sigma\circ\tau\circ \pi$ has three bonds at least.
\end{prop}

However, what the authors actually proved in their paper \cite{EE} is the following proposition.
\begin{prop}\label{propEE}
Let $\pi$ be any permutation on $[n]$ other than the reverse. Then $\pi_\circ^\circ$ contains a permutation $\bar\pi$ on $[n]^0$ with $\bar\pi_0=0$ having the following properties. There are block transpositions $\sigma$ and $\tau$ such that either $\bar\pi\circ[0\,\sigma]\circ[0\,\tau],$ or $[0\,\sigma]\circ\bar\pi\circ[0\,\tau],$ or $[0\,\sigma]\circ[0\,\tau]\circ\bar\pi$ has three bonds at least.
\end{prop}

An important consequence of Proposition \ref{propEE} is the following result.
\begin{cor}
\label{kitty}
Let $\pi$ be any permutation on $[n]$ other than the reverse. Then there exist a permutation $\rho$ torically equivalent to $\pi$ and block transpositions $\sigma$ and $\tau$ such that either $\rho\circ\sigma\circ\tau,$ or $\sigma\circ\rho\circ\tau,$ or $\sigma\circ\tau\circ\rho$ has three bonds at least.
\end{cor}

Assume that the first case of Proposition \ref{propEE} occurs. Observe that  $\bar\pi\circ[0\,\sigma]\circ[0\,\tau]=[0\,\rho]$ for some permutation $\rho$ on $[n]$. Since $[0\,\rho]$ has as many
bonds as $\rho$ does, Corollary \ref{kitty} holds. The authors showed in \cite{EE} that Proposition \ref{wrong} together with other arguments yield the following upper bound on the block transposition diameter.
\begin{thm}
\label{EEmainbis}{\rm{[Eriksson's bound]}}
For $n\geq 9$,
$$d(n)\leq \left\lfloor\frac{2n-2}{3}\right\rfloor.$$
\end{thm}

Actually, as it was pointed out by Elias and Hartman in \cite{EH}, Proposition \ref{propEE} only ensures the weaker bound
$\left\lfloor\frac{2n}{3}\right\rfloor.$ Nevertheless, Proposition \ref{wrong} and Corollary \ref{kitty} appear rather similar, indeed they coincide in the toric class. This explains why Eriksson's upper bound still holds; see Section \ref{tiger}. In this paper we complete the proof of Eriksson's bound. We show indeed that Eriksson's bound follows from Proposition \ref{propEE} together with Theorem \ref{torically}. We also revisit the proof of Proposition \ref{propEE} giving more technical details and filling some gaps.

\section{Proof of Theorems \ref{torically} and \ref{principle}}\label{invariance}
The essential tool of our proofs is the following result.
\begin{lem}{\rm{[Shifting lemma]}}
\label{lemc19ott2013}
Let $\sigma(i,j,k)$ be any block transposition on $[n]$. Then, for every integer $r$ with $0\leq r\leq n$, there exists a block transposition $\sigma(i',j',k')$ on $[n]$ such that
$$[0\,\sigma(i,j,k)]\circ\alpha^r=\alpha^{[0\,\sigma(i,j,k)]_r}\circ[0\,\sigma(i',j',k')].$$
\end{lem}
\begin{proof}
Let $\sigma=[0\,\sigma(i,j,k)]$. Since $0\leq i<j<k\leq n$, one of the following four cases can only occur:
\begin{remunerate}
\item$0\leq i-r<k-j+i-r<k-r\leq n$,
\item$0\leq k-j+i-r<k-r<n+1+i-r\leq n$,
\item$0\leq k-r<n+1+i-r<n+1+k-j+i-r\leq n$,
\item$0\leq n+1+i-r<n+1+k-j+i-r<n+1+k-r\leq n$.
\end{remunerate}

If the hypothesis in case 1 is satisfied, we have, by (\ref{cuppoints}) and (\ref{alpha_powers}),
$$\begin{array}{l}
(\sigma\circ\alpha^r)_0=r,\quad(\sigma\circ\alpha^r)_{i-r}=i,\quad(\sigma\circ\alpha^r)_{i+1-r}=j+1,\quad(\sigma\circ\alpha^r)_{k-j+i-r}=k,\\
(\sigma\circ\alpha^r)_{k-j+i+1-r}=i+1,\quad(\sigma\circ\alpha^r)_{k-r}=j,\quad(\sigma\circ\alpha^r)_{k+1-r}=k+1,\\
(\sigma\circ\alpha^r)_{n-r}=n,\quad(\sigma\circ\alpha^r)_{n+1-r}=0,\quad(\sigma\circ\alpha^r)_n=r-1,
\end{array}$$where superscripts and indices are taken mod $n+1$. By (\ref{function}), $(\sigma\circ\alpha^r)_{t+1}=(\sigma\circ\alpha^r)_{t} +1$ in the intervals $[0,i-r],\,[i-r+1,k-j+i-r],\,[k-j+i-r+1,k-r],\,[k-r+1,n-r],\,[n-r+1,n]$. From this
$$(\alpha^{-r}\circ\sigma\circ\alpha^r)_{t+1}=(\alpha^{-r}\circ\sigma\circ\alpha^r)_{t} +1$$in the intervals $[0,i-r],\,[i-r+1,k-j+i-r],\,[k-j+i-r+1,k-r],\,[k-r+1,n]$, and
$$\begin{array}{l}
(\alpha^{-r}\circ\sigma\circ\alpha^r)_0=0,\quad(\alpha^{-r}\circ\sigma\circ\alpha^r)_{i-r}=i-r,\quad(\alpha^{-r}\circ\sigma\circ\alpha^r)_{i+1-r}=j+1-r,\\
(\alpha^{-r}\circ\sigma\circ\alpha^r)_{k-j+i-r}=k-r,\quad(\alpha^{-r}\circ\sigma\circ\alpha^r)_{k-j+i+1-r}=i+1-r,\\
(\alpha^{-r}\circ\sigma\circ\alpha^r)_{k-r}=j-r,\quad(\alpha^{-r}\circ\sigma\circ\alpha^r)_{k+1-r}=k+1-r,\\
(\alpha^{-r}\circ\sigma\circ\alpha^r)_n=n.
\end{array}$$Since $\sigma_r=r$ and $0\leq i-r<j-r<k-r\leq n$, the statement follows in case 1 from (\ref{function}) and (\ref{cuppoints}) with $i'=i-r,\,j'=j-r,\,k'=k-r$.

Suppose $0\leq k-j+i-r<k-r<n+1+i-r\leq n$. By (\ref{cuppoints}) and (\ref{alpha_powers}),
$$\begin{array}{l}
(\sigma\circ\alpha^r)_0=\sigma_r,\quad(\sigma\circ\alpha^r)_{k-j+i-r}=k,\quad(\sigma\circ\alpha^r)_{k-j+i+1-r}=i+1,\\
(\sigma\circ\alpha^r)_{k-r}=j,\quad(\sigma\circ\alpha^r)_{k+1-r}=k+1,\\
(\sigma\circ\alpha^r)_{n-r}=n,\quad(\sigma\circ\alpha^r)_{n+1-r}=0,\\
(\sigma\circ\alpha^r)_{n+1+i-r}=i,\quad(\sigma\circ\alpha^r)_{n+2+i-r}=j+1,\\
(\sigma\circ\alpha^r)_n=\sigma_r-1,
\end{array}$$
where superscripts and indices are taken mod $n+1$. Therefore, we obtain $\sigma_r-j-2=n-(n+2+i-r)$ hence $\sigma_r=-(i-j-r)$, and
$$\begin{array}{l}
(\alpha^{i-j-r}\circ\sigma\circ\alpha^r)_0=0,\quad(\alpha^{i-j-r}\circ\sigma\circ\alpha^r)_{k-j+i-r}=k+i-j-r,\\
(\alpha^{i-j-r}\circ\sigma\circ\alpha^r)_{k-j+i+1-r}=n+2+2i-j-r,\\
(\alpha^{i-j-r}\circ\sigma\circ\alpha^r)_{k-r}=n+1+i-r,\\
(\alpha^{i-j-r}\circ\sigma\circ\alpha^r)_{k+1-r}=k+1+i-j-r,\\
(\alpha^{i-j-r}\circ\sigma\circ\alpha^r)_{n+1+i-r}=n+1+2i-j-r,\\
(\alpha^{i-j-r}\circ\sigma\circ\alpha^r)_{n+2+i-r}=n+2+i-r,\quad(\alpha^{i-j-r}\circ\sigma\circ\alpha^r)_n=n.
\end{array}
$$(\ref{function}) gives $(\sigma\circ\alpha^r)_{t+1}=(\sigma\circ\alpha^r)_{t} +1$ in the intervals $$[0,k-j+i-r],\,[k-j+i-r+1,k-r],\,[k-r+1,n-r],\,[n-r+1,n+1+i-r],\,[n+2+i-r,n].$$Therefore, $(\alpha^{i-j-r}\circ\sigma\circ\alpha^r)_{t+1}=(\alpha^{i-j-r}\circ\sigma\circ\alpha^r)_{t} +1$ in the above intervals. Since $k-r=n+1+i-r-(n+1+2i-j-r)+k-j+i-r$, the statement follows in case 2 from (\ref{function}) and (\ref{cuppoints}) with $i'=k-j+i-r,\,j'=n+1+2i-j-r,\,k'=n+1+i-r$.

Assume $0\leq k-r<n+1+i-r<n+1+k-j+i-r\leq n$. By (\ref{cuppoints}) and (\ref{alpha_powers}),
$$\begin{array}{l}
(\sigma\circ\alpha^r)_0=\sigma_r,\quad(\sigma\circ\alpha^r)_{k-r}=j,\quad(\sigma\circ\alpha^r)_{k+1-r}=k+1,\quad(\sigma\circ\alpha^r)_{n-r}=n,\\
(\sigma\circ\alpha^r)_{n+1-r}=0,\quad(\sigma\circ\alpha^r)_{n+1+i-r}=i,\quad(\sigma\circ\alpha^r)_{n+2+i-r}=j+1,\\
(\sigma\circ\alpha^r)_{n+1+k-j+i-r}=k,\quad(\sigma\circ\alpha^r)_{n+2+k-j+i-r}=i+1,\quad(\sigma\circ\alpha^r)_n=\sigma_r-1,
\end{array}$$
where superscripts and indices are taken mod $n+1$. It is straightforward to check that $\sigma_r-2-i=n-(n+2+k-j+i-r)$ hence $\sigma_r=-(k-j-r)$. Furthermore, the following relations
$$\begin{array}{l}
(\alpha^{k-j-r}\circ\sigma\circ\alpha^r)_0=0,\quad(\alpha^{k-j-r}\circ\sigma\circ\alpha^r)_{k-r}=k-r,\\
(\alpha^{k-j-r}\circ\sigma\circ\alpha^r)_{k+1-r}=2k-j-r+1,\\
(\alpha^{k-j-r}\circ\sigma\circ\alpha^r)_{n+1+i-r}=n+1+k-j+i-r,\\
(\alpha^{k-j-r}\circ\sigma\circ\alpha^r)_{n+2+i-r}=k-r+1,\\
(\alpha^{k-j-r}\circ\sigma\circ\alpha^r)_{n+1+k-j+i-r}=2k-j-r,\\
(\alpha^{k-j-r}\circ\sigma\circ\alpha^r)_{n+2+k-j+i-r}=n+2+k-j+i-r,\quad(\alpha^{k-j-r}\circ\sigma\circ\alpha^r)_n=n.
\end{array}$$hold. By (\ref{function}), $(\sigma\circ\alpha^r)_{t+1}=(\sigma\circ\alpha^r)_{t} +1$ in the intervals $[0,k-r],\,[k-r+1,n-r],\,[n-r+1,n+1+i-r],\,[n+2+i-r,n+1+k-j+i-r],\,[n+2+k-j+i-r,n]$. Then we obtain $(\alpha^{k-j-r}\circ\sigma\circ\alpha^r)_{t+1}=(\alpha^{k-j-r}\circ\sigma\circ\alpha^r)_{t} +1$ in the same intervals. Since $n+1+i-r=n+1+k-j+i-r-(2k-j-r)+k-r$, the statement follows in case 3 from (\ref{function}) and (\ref{cuppoints}) with $i'=k-r,\,j'=2k-j-r,\,k'=n+1+k-j+i-r$.

To deal with case 4, it is enough to use the same argument of case 1 replying $i-r$ with $n+1+i-r$, $j-r$ with $n+1+j-r$, and $k-r$ with $n+1+k-r$. Hence the statement follows with $i'=n+1+i-r,\,j'=n+1+j-r,\,k'=n+1+k-r$ and $\sigma_r=r$.\qquad\end{proof}

Now we are in a position to prove Theorem \ref{torically}. Take two torically equivalent permutations $\pi$ and $\pi'$ on $[n]$.
Let $d(\pi)=k$, and let $\pi=\sigma_1\circ\cdots \circ \sigma_k$ with $\sigma_1,\ldots,\sigma_k\in S$. Then
$$[0\,\pi]=[0\,\sigma_1] \circ \cdots \circ [0\,\sigma_k].$$By (\ref{toric_zero}), there exists an integer $r$ with $0\le r \le n$ such that
\begin{equation}
\label{eq323jan}
[0\,\pi']=\alpha^{-\pi_r}\circ [0\,\sigma_1] \circ \cdots \circ [0\,\sigma_k] \circ \alpha^r.
\end{equation}The Shifting Lemma applied to $[0\,\sigma_k]$ allows us to shift $\alpha^r$ to the left in (\ref{eq323jan}), in the sense that $[0\,\sigma_k]\circ\alpha^{r}$ is replaced by $\alpha^t\circ[0\,\rho_k]$ with $t=-(\sigma_k)_r$. Repeating this $k$ times yields
$$[0\,\pi']=\alpha^s\circ[0\,\rho_1]\circ \cdots \circ [0\,\rho_k],$$ for some integer $0\le s \le n$. Actually, $s$ must be zero as $\alpha^s$ can fix $0$ only for $s=0$. Then $\pi'=\rho_1\circ\cdots\circ\rho_k$, and there exist $r_1,\cdots,r_k$ integers with $0\le r_i \le n$ such that
$$\pi'=\Omega_{r_1}(\sigma_1)\circ\Omega_{r_2}(\sigma_2)\circ\cdots\circ\Omega_{r_k}(\sigma_k).$$
Therefore, $d(\pi')\le k=d(\pi)$. By inverting the roles of $\pi$ and $\pi'$, we also obtain $d(\pi)\le d(\pi')$. Hence the claim follows.

To prove Theorem \ref{principle}, it suffices to show that $d(\nu^{-1}\circ\pi)=d(\mu^{-1}\circ\varpi)$, by the left-invariance of the block transposition distance; see Proposition \ref{distance}. Let $d(\nu^{-1}\circ\pi)=h$, and let $\sigma_1,\ldots,\sigma_h\in S$ such that
\begin{equation}\label{a1}
[0\,\nu^{-1}]\circ[0\,\pi]=[0\sigma_h]\circ\cdots \circ [0\,\sigma_1].
\end{equation}Since $\Omega_r$ takes $\pi$ to $\varpi$ and $\nu$ to $\mu$, we obtain $\alpha^{-r}\circ[0\,\varpi]\circ\alpha^{-s}=[0\,\pi]$ and $\alpha^{t}\circ[0\,\mu^{-1}]\circ\alpha^{r}=[0\,\nu^{-1}],$ where $r$ is an integer with $0\leq r \leq n$, $t=-\nu_r$, and $s=-\pi_r$, by Definition \ref{defOmega}. Hence,\begin{equation}\label{a2}
[0\,\nu^{-1}]\circ[0\,\pi]=\alpha^{t}\circ[0\,\mu^{-1}]\circ\alpha^{r}\circ\alpha^{-r}\circ[0\,\varpi]\circ\alpha^{-s}.
\end{equation}Then $[0\,\mu^{-1}]\circ[0\,\varpi]=\alpha^{-t}\circ[0\,\sigma_h]\circ\cdots \circ[0\,\sigma_1]\circ\alpha^s$ follows from (\ref{a1}) and (\ref{a2}). By the Shifting Lemma, this may be reduced to $[0\,\mu^{-1}]\circ[0\,\varpi]=[0\,\sigma'_h]\circ\cdots\circ[0\,\sigma'_1]\circ\alpha^q$ for some integer $q$ with $0\leq q\leq n$ and $\sigma'_1,\ldots,\sigma'_h\in S$. As we have seen in the proof of Theorem \ref{torically}, $q$ must be $0$, and $d(\mu^{-1}\circ\varpi)=d(\nu^{-1}\circ\pi)$. Therefore, the claim follows.

\section{Criteria for the existence of a $2$-move}

The proof of Proposition \ref{propEE} is constructive. For any permutation $\pi$, our algorithm will provide two block transpositions $\sigma$ and $\tau$ together with a permutation $\bar\pi\in \pi^\circ_\circ$ so that either $\bar\pi\circ[0\,\sigma]\circ[0\,\tau],$ or $[0\,\sigma]\circ\bar\pi\circ[0\,\tau],$ or $[0\,\sigma]\circ[0\,\tau]\circ\bar\pi$ has three bonds at least. For the seek of the proof, $\pi$ is assumed to be bondless, otherwise all permutations in its toric class has a bond, and two more bonds by two block transpositions can be found easily.

A \emph{k-move (to the right)} of $\bar\pi\in{\rm{Sym_n^0}}$ is a block transposition $\bar\sigma$ on $[n]^0$ such that $\bar\pi\circ\bar\sigma$ has (at least) $k$ more bonds than $\bar\pi$. A block transposition $\bar\sigma$ is a \emph{k-move to the left} of $\bar\pi$ if $\bar\sigma\circ\bar\pi$ has (at least) $k$ more bonds than $\bar\pi$.
\begin{crt}\label{2-move to the right}
A $2$-move of $\bar\pi\in{\rm{Sym_n^0}}$ exists if one of the following holds:
\begin{romannum}
\item $\bar\pi=[\cdots x\cdots y\,\overline x\cdots \overline y\cdots],$
\item $\bar\pi=[\cdots x\cdots \underline x\,\overline x\cdots].$
\end{romannum}
\end{crt}
\begin{proof}
Each of the following block transpositions:
$$x|\cdots y|\overline x\cdots |\overline y,\quad|x|\cdots \underline x|\overline x.$$gives two new bonds for (i) and (ii), respectively.\qquad\end{proof}

In a permutation an ordered triple of values $x\cdots y\cdots z$ is \emph{positively oriented} if either $x<y<z$, or $y<z<x$, or $z<x<y$ occurs.
\begin{crt}\label{2-move to the left}
Let $\bar\pi$ be a permutation on $[n]^0$. A $2$-move to the left of $\bar\pi$ occurs if one of the following holds.
\begin{romannum}
\item $\bar\pi=[\cdots x\,y\dots z\,\overline x\cdots]$ and $x,y,z$ are positively oriented,
\item $\bar\pi=[\cdots x\,y\,\overline x\cdots].$
\end{romannum}In particular, the following block transpositions on $[n]^0$ create a $2$-move to the left of $\bar\pi$.
\begin{remunerate}
\item[]For {\rm{(i)}},
\item $\bar\sigma(x,x+z-y+1,z)$ if $x<y<z$,
\item $\bar\sigma(y-1,y-1+x-z,x)$ if $y<z<x$,
\item $\bar\sigma(z,z+y-1-x,y-1)$ if $z<x<y$.
\item[]For \rm{(ii)},
\item $\bar\sigma(x,x+1,y)$ if $x<y$,
\item $\bar\sigma(y-1,x-1,x)$ if $y<x$.
\end{remunerate}
\end{crt}
\begin{proof}
Let $\bar\sigma=\bar\sigma(a,b,c)$ for any $a,\,b,\,c$ with $-1\leq a<b<c\leq n$.  Then
$$\bar\sigma_t =\left\{\begin{array}{ll}
t, &  0\leq t\leq a\quad c+1\leq t\leq n,\\
t+ b- a, & a+1\leq t\leq c- b+ a,\\
t+ b- c, & c-b+a+1\leq t\leq c
\end{array}\right.$$follows from (\ref{function}). In particular, by (\ref{cuppoints}),
\begin{equation}\label{tras}
\begin{array}{l}
\bar\sigma_a= a,\quad\bar\sigma_{a+1}=b+1,\quad\bar\sigma_{a+c-b}=c,\\
\bar\sigma_{a+c-b+1}=a+1,\quad\bar\sigma_{c}=b,\quad\bar\sigma_{c+1}=c+1.
\end{array}
\end{equation}Our purpose is to determinate $a,b,c$ so that $\bar\sigma$ is a $2$-move to the left of $\pi$. For this, $\bar\sigma$ must satisfy the following relations:
\begin{equation}\label{2-move}
\bar\sigma_y=\bar\sigma_x+1\quad\bar\sigma_{\overline x}=\bar\sigma_z+1.
\end{equation}

If the hypothesis in case 1 is satisfied, take $x$ for $a$. By (\ref{tras}), we obtain both $\bar\sigma_x=a$ and $\bar\sigma_{\overline x}=b+1$. This together with (\ref{2-move}) gives $\bar\sigma_y=a+1$ and $\bar\sigma_z=b$. By (\ref{tras}), we get
$$\left\{\begin{array}{l}
a=x\\
c=z\\
b=a+c-y+1.
\end{array}\right.$$Since $\pi$ is bondless, here $x<y-1$ may be assumed. This together with $y<z$ give $x<x+z-y+1<z$, whence the statement in case 1 follows. To deal with case 4 it is enough to use the same argument after switching $z$ and $y$.

If the hypothesis in case 2 is satisfied, take $\bar\sigma_x=b$ and $\bar\sigma_z=c$. This choice together with (\ref{2-move}) gives $\bar\sigma_y=b +1$ and $\bar\sigma_{\overline x}=c+1$. By (\ref{tras}), we have
$$\left\{\begin{array}{l}
a=y-1\\
c=x\\
b=a+c-z.\\
\end{array}\right.$$Since $y-1<y-1+x-z<x$, the statement in case 2 follows. Case 5 may be settled with the same argument after switching $z$ and $y$.

In case 3, let $\bar\sigma_x=c$ and $\bar\sigma_z =a$. This together with (\ref{2-move}) gives $\bar\sigma_y=c +1$ and $\bar\sigma_{\overline x}=a+1$. By (\ref{tras}), we obtain
$$\left\{\begin{array}{l}
a=z\\
c=y-1\\
b=a+c-x.\\
\end{array}\right.$$Since $z<z+y-1-x<y-1$, the statement holds.\qquad\end{proof}

\begin{rem}\label{fix}
\emph{Note that the block permutation on $[n]^0$ appearing in cases 1,3, and 4 of Criterion \ref{2-move to the left} fixes $0$. In cases 2 and 5, this occurs if and only if $y\neq 0$.}
\end{rem}

For the proof of Proposition \ref{propEE}, we begin by constructing the required moves for reducible permutations.

\section{Reducible case}
A permutation $\pi$ on $[n]$ is \emph{reducible} if for some $k$ with $0<k<n$ the segment $0\cdots \pi_k$ contains all values $0,\ldots,k$ while the segment $\pi_k\cdots n$ contains all values $k,\ldots, n$. In particular, $\pi_k=k$ is required. It is crucial {\bf{to note}} that a reducible permutation collapses into a smaller permutation by erasing the segment $\pi_k\cdots \pi_n$. If a reverse permutation is produced in this way, we proceed by contracting the segment $0\cdots \pi_k$ to $0$. There may be that both contractions produce a reverse permutation. This only occurs when
\begin{equation}\label{red}
[0\,\pi]=[0\,k-1\, k-2 \cdots 1 \, k \, n \, n-1 \cdots k+1].
\end{equation}
After carrying out the $1$-move $k-1| k-2 \cdots 1 \, k \,n| n-1 \cdots  k+1|,$
Criterion \ref{2-move to the left}.3 applies to $k-1 \,n-1 \cdots 1 \, k$ whence Proposition \ref{propEE} follows in this case.

To investigate the other cases we show that a permutation of the form (\ref{red}) occurs after a finite number of steps. For this purpose, let $[0\,\pi]^0=[0\,\pi]$, and let $[0\,\pi]^l=[0\,\pi_1^l \cdots \pi_{n_l}^l]$ for any integer $\ell\ge 0$. First, we prove that reducing $[0\,\pi]^l$ diminishes the length of $[0\,\pi]^l$ by at least three. Observe that a reduced permutation $[0\,\pi]^l$ is bondless since $[0\,\pi]$ is bondless. As $k_l<n_l$, erasing the segment $\pi^l_{{n_l}-1}\,\pi^l_{n_l}$ gives $\pi^l_{n_l-1}=k_l$ and $\pi^l_{n_l}=k_l+1.$ If we contract $0\,\pi_1^l\,\pi_2^l$ into $0$, we obtain $k_l=\pi^l_2=2$ and $\pi^l_1=1.$ Actually, none of these possibilities can occur as $[0\,\pi]^l$ is bondless. Nevertheless, we are able to reduce the length of $[0\,\pi]^l$ by exactly three. For instance, let $[0\,\pi]^l=[0\,2\,1\,3\cdots]$. After contracting $l\leq\lfloor n/3\rfloor$ times the following four cases:
\begin{romannum}
\item $[0\,\pi]^l=[0]$,
\item $[0\,\pi]^l=[0\,1]$,
\item $[0\,\pi]^l=[0\,2\,1]$,
\item $[0\,\pi]^l=[0\,1\,2]$
\end{romannum}
hold. In case (i) and (ii) we have $k_{l-1}=1$ and $k_{l-1}=2$, respectively. As $[0\,\pi]^l$ is bondless, only case (iii) occurs, a contradiction, since collapsing any segment of $[0\,\pi]^{l-1}$ does not produce a reverse permutation. Therefore, the assertion follows.

\section{Irreducible case}

It remains to prove Proposition \ref{propEE} when $\pi$ is bondless and irreducible. We may also assume that no permutation $\bar\pi\in\pi_\circ^\circ$ satisfies either Criterion \ref{2-move to the right} or Criterion \ref{2-move to the left} with a block transposition fixing $0$. Otherwise, getting a further $1$-move is trivial.

Up to toric equivalence, choose $\bar\pi$ fulfilling the minimality condition on $0\cdots 1$, that is, the shortest sequence $m\cdots \overline{m}$ in $\bar\pi$ occurs for $m=0$. To prove that such a permutation exists, we start with $[0\,\pi]=[0\cdots \pi_u\cdots \pi_v \cdots n]$, where $\pi_u=m, \pi_v=\overline{m}$. Note that $\pi$ is torically equivalent to $\pi'$, where $\bar\pi=[0\,\pi']$ is defined by $\bar\pi_x=\pi_{x+u}-m$ for every integer $x$ with $0\le x \le n$, and the indices are taken mod $n+1$. Then $\bar\pi_0=0$ and $\bar\pi_{v-u}=\pi_v-m=1$.

We begin by observing that the minimality condition on $0\cdots 1$ always rules out the case $\bar\pi=[\cdots x\cdots \overline x\cdots 1\cdots]$. The absence of bonds rules out the extremal case $\bar\pi=[0\,1\cdots]$, while the absence of a $2$-move fixing $0$ makes it possible to avoid $\bar\pi=[0\,x_1\,1\cdots]$ by applying Criterion \ref{2-move to the left}.4 to $0\,x_1\,1$. Therefore, we may write $\bar\pi$ in the form $[0\,x_1\cdots x_l\,1 \cdots]$ with $l\geq 2$. Note that $x_1>x_l$, otherwise Criterion \ref{2-move to the left}.1 applies to $0\,x_1\cdots x_l\, 1$, whence $\overline x_1$ is on the right of $1$ when $x_1\neq n$. Now one of the $1$-move listed below
\begin{romannum}
\item $0|x_1\,x_2\cdots x_l|1\cdots|\overline x_1$ if $x_1\neq n$,
\item $0|x_1\cdots x_l|1\cdots x_n|$ if $x_1=n$ and $x_n\neq 1$,
\item $0|x_1\cdots x_l|1|$ if $x_1=n$ and $x_n= 1$
\end{romannum}
turns $\bar\pi$ in one of the following forms:
\begin{remunerate}
\item $[0\cdots x_1\, x_2\cdots x_l\,\overline x_1\cdots]$ if $x_1\neq n$ and $l\neq 2$,
\item $[0\cdots x_1\,x_2\,\overline x_1\cdots]$ if $x_1\neq n$ and $l=2$,
\item $[0\cdots x_1\,x_2\cdots x_l]$ if $x_1=n,\,x_n\neq 1$ and $l\neq 2$,
\item $[0\cdots x_1\,x_2]$ if $x_1=n,\,x_n\neq 1$ and $l=2$,
\item $[0\,1\,x_1\,x_2\cdots x_l]$ if $x_1=n,\,x_n= 1$ and $l\neq2$,
\item $[0\,1\,x_1\, x_2]$ if $x_1=n,\,x_n= 1$ and $l=2$.
\end{remunerate}
Unless $x_1>x_2>x_l$, Proposition \ref{propEE} holds. In fact, there exists a permutation in $\pi_\circ^\circ$ that satisfies one of the hypotheses of Criterion \ref{2-move to the left}. More precisely, we may apply either Criterion \ref{2-move to the left}.2 or Criterion \ref{2-move to the left}.3 in case 1 and Criterion \ref{2-move to the left}.2 in case 3. For $l=2$, the statement follows from Criterion \ref{2-move to the left}.5. Some block transpositions in Criterion \ref{2-move to the left}.2 and \ref{2-move to the left}.5 may not fix $0$. This cannot actually occur, since we use block transpositions on $[n]^0$ of the form $[0\,\sigma(i,j,k)]$ in all cases; see Remark \ref{fix}.

Therefore, we may assume $\bar\pi=[0\,x_1\,x_2\cdots x_l\,1\cdots]$ with $x_1>x_2>x_l$. Two cases are treated separately according as $x_2=x_1-1$ or $x_2<x_1-1$.

\subsection{Case $x_2= x_1-1$}\label{casee}

If $\bar\pi=[0\,x_1\,x_2\cdots x\cdots1\cdots \overline x\cdots]$ occurs for some $x$, then the $1$-move $$x_1|x_2\cdots x|\cdots 1\cdots|\overline x$$ turns $\bar\pi$ into $[0\,x_1\cdots 1\cdots x_2\cdots]$. After that, the existence of a $2$-move is ensured by Criterion \ref{2-move to the right}, whence Proposition \ref{propEE} holds.

Therefore, we may assume that if $x$ ranges over $x_3,x_4,\ldots,x_i,\ldots x_l$, then $\overline x$ is on the left of $x$. At each stage two cases arise depending upon whether $\overline x_i=x_{i-1}$ or $x_i=n$, where $l\neq i\neq 2$.

\subsubsection{Case $\overline x_i=x_{i-1}$}\label{imp}

Note that $x_1\,x_2\cdots x_l$ is a reverse consecutive sequence, and $n$ is on the right of $1$. As $\pi$ is bondless, two cases arise according as either $1<x_n<x_l$ or $x_1<x_n<n$.

In the former case, carrying out the $1$-move $|\overline {x}_l\,x_l\,1\cdots n|\cdots x_n|$, the resulting permutation is $[\cdots x_n\,\overline {x}_l\,x_l\,1\cdots]$. In the latter case, use the $1$-move $|x_1\cdots|1\cdots n|$ to obtain $[0\cdots n\,x_1\cdots x_n]$. In both cases Proposition \ref{propEE} follows from Criterion \ref{2-move to the left}.2, applied to a block transposition that fixes $0$; see Remark \ref{fix}. More precisely, let $\bar\pi\circ[0\,\sigma]$ be the permutation obtained in both cases, and let $\bar\pi\circ[0\,\sigma]\circ\alpha^r$ with $1\leq r\leq n$ be the permutation that satisfies the hypothesis of Criterion \ref{2-move to the left}.2. Then $[0\,\tau]\circ\bar\pi\circ[0\,\sigma]\circ\alpha^r$ has three bonds at least. By the Shifting Lemma \ref{lemc19ott2013}, there exists an integer $s$ with $1\leq s\leq n$ and a block transposition $\sigma'$ on $[n]$ such that $$[0\,\tau]\circ\bar\pi\circ[0\,\sigma]\circ\alpha^r=[0\,\tau]\circ\bar\pi\circ\alpha^s\circ[0\,\sigma'].$$ Since $\bar\pi\circ\alpha^s\in \pi_\circ^\circ$, Proposition \ref{propEE} follows.

\subsubsection{Case $x_1=n$}

In this case there exists $k$ with $2\leq k\leq n-2$ such that
$$\bar\pi=[0\,n\,n-1\cdots n-(k-2)\,n-(k-1)\,1\cdots n-k \cdots].$$Since $\pi$ is not the reverse permutation, $2$ is on the right of $1$ whence $2\leq k\leq n-2$. So two cases arise depending on the position of $2$ with respect to $n-k$.

If $2$ is on the left of $n-k$, then the $1$-move $|n-(k-1)\,1\,y\cdots |2\cdots n-k|$ turns $\bar\pi$ into $[\cdots n-(k-2)\,2\cdots 1\,y\cdots]$. As all integers $x$ with $n-(k-2)\leq x\leq n$ are in $0\cdots1$, this yields $y<n-(k-2)$. So Criterion \ref{2-move to the left}.1 applies to a permutation in the circular class of $[\cdots n-(k-2)\,2\cdots 1\,y\cdots]$, and the claim follows as in section \ref{imp}.

If $2$ is on the right of $n-k$, consider the $1$-move $$|n-(k-2)\, n-(k-1)\, 1|\cdots n-k\cdots z|2.$$ If $z=n-k$, then the above transposition takes our permutation to $$[\cdots n-k\,n-(k-2)\,n-(k-1)\cdots].$$ The existence of a $2$-move is ensured by Criterion \ref{2-move to the left}.4. Otherwise $z<n-k$, and Criterion \ref{2-move to the left}.2 applies to a permutation in the circular class of $$[\cdots z\,n-(k-2)\,n-(k-1)\,1\cdots]$$and a block transposition that fixes $0$; see Remark \ref{fix}. Hence the claim follows as in section \ref{imp}.

\subsubsection{Case $x_i=n$}

As we have seen in Case \ref{casee}, $\bar x$ is on the left of $x$ for every $x$ in $0\cdots 1$. Therefore, when $x_i=n$, $x_{j-1}=\overline{x}_j$ for $1\leq j\leq i-1$, but it does not necessarily holds for all $j$ with $i\leq j\leq l$. However, there exists $h$ with $1\leq h \leq n-1-x_1$ such that for each $x\neq 0$ on the left of $x_l$ either $\overline{x}_l\leq x \leq x_1$ or $n-(h-1)\leq x \leq n$ occurs. Both these subsequences are decreasing by our minimality condition on $0\cdots 1$.

First, suppose the existence of $k$ with $3 \leq k\leq h$ so that
$$\bar\pi=[0\,x_1\,\underline {x_1}\cdots x_t\,\underline {x_t}\,n\,n-1\,n-2\cdots n-(k-3)\,n-(k-2)\,n-(k-1)\,\underline{\underline {x_t}}\cdots 1\cdots],$$
where $\overline {x}_l\leq \underline {x_t} \leq \underline {x_1}$ and $\underline{\underline x}$ stands for $\underline y$ with $y=\underline x$. Now one of the following $1$-move:
$$\begin{array}{ll}
x_t|\underline {x_t}\,n\cdots n-(k-3)|n-(k-2)\,n-(k-1)\,\underline{\underline{x_t}}|,& k>3,\\
x_t|\underline {x_t}\,n|n-1\,n-2\,\underline{\underline{x_t}}|,& k=3
\end{array}$$turns $\bar\pi$ into $[\cdots x_t\,n-(k-2)\,n-(k-1)\,\underline{\underline{x_t}}\cdots]$. Therefore, Criterion \ref{2-move to the left}.2 applies to a permutation in the circular class of $[\cdots x_t\,n-(k-2)\,n-(k-1)\,\underline{\underline{x_t}}\cdots]$ and a block transposition fixing $0$; see Remark \ref{fix}. Hence the assertion follows as in {\bf{Section}} \ref{imp}.

Note that case $\bar\pi=[\cdots\underline {x_t}\,n\,n-1\,\underline{\underline {x_t}}\cdots 1\cdots]$ does not occur. In fact, the existence of a $2$-move of a permutation in the toric class $\pi_\circ^\circ$ and a block transposition fixing $0$ is ensured by Criterion \ref{2-move to the left}.2 and Remark \ref{fix}. Therefore, we may assume that $$\bar\pi=[\cdots\underline {x_t}\,n\,\underline{\underline {x_t}}\cdots].$$ Now a $2$-move of a permutation in the toric class of $\pi^\circ_\circ$ with a block transposition fixing $0$ is ensured by Criterion \ref{2-move to the left}.4, a contradiction.

\subsection{Case $x_2< x_1-1$}

If $\underline {x_1}$ is on the right of $1$, then there exists a $2$-move of $[0\,x_1\cdots 1 \cdots \underline {x_1}\cdots]$ by Criterion \ref{2-move to the right}, a contradiction. Therefore, $\underline {x_1}$ is on the left of $1$. We look for the biggest integer $k$ with $2\leq k\leq l-1$ such that
\begin{equation}\label{diseq}
x_1-(k-1)>x_2-(k-2)>\cdots>x_i-(k-i)>\cdots> x_{k-1}-1>x_k>x_l
\end{equation}hold. Note that (\ref{diseq}) holds for $k=2$ by $x_1-1>x_2>x_l$. Suppose that $\overline {x}_k$ is on the left of $1$ with $x_i=\overline {x}_k$ for some $i$. Then $1\leq i\leq k-1$ and $x_i-1\geq x_i-(k-i)>x_k$, a contradiction. Therefore, $\overline {x}_k$ must be on the right of $1$. The $1$-move $0|x_1\cdots x_l|1\cdots |\overline {x}_k$ turns $\bar\pi$ into $[0\,1\cdots x_k\,x_{k+1}\cdots x_l\,\overline {x}_k\cdots],$ and the following three possibilities arise:
\begin{enumerate}
\item $x_l<x_k<x_{k+1}$,
\item $x_{k+1}<x_l<x_k$,
\item $x_l<x_{k+1}<x_k$.
\end{enumerate}Proposition \ref{propEE} follows from Criterion \ref{2-move to the left}.3 in case 1 and from Criterion \ref{2-move to the left}.2 in case 2, applied to a block transposition fixing $0$; see Remark \ref{fix}.

In the remaining case, adding $1$ to each side in (\ref{diseq}) gives $x_{i-1}-(k-i)> x_{k-1},$ where $1<i<k$. If $\overline {x}_{k-1}$ is on the left of $1$ and $x_{i-1}=\overline {x}_{k-1}$ for some $i>1$, then $x_{i-1}-1> x_{i-1}-(k-i)> x_{k-1}$, a contradiction. If $\overline {x}_{k-1}$ is on the left of $1$ and $x_1=\overline {x}_{k-1}$, subtracting $1$ from each side in (\ref{diseq}) gives
$$x_1-k>x_2-(k-1)>\cdots>x_i-(k+1-i)>\cdots> x_{k-1}-2>x_k-1.$$Here $x_k-1>x_{k+1}$ cannot actually occur by our maximality condition on $k$. Therefore, $x_{k+1}=x_k-1$. The $1$-move $|x_{k-1}\,x_k|\underline{x_k}\cdots1\cdots|\overline {x}_k$ turns $\bar\pi$ into $[0\,\overline {x}_{k-1}\cdots1\cdots x_{k-1}]$, and Proposition \ref{propEE} follows from Criterion \ref{2-move to the right}. Here we consider $\overline {x}_{k-1}$ to be on the right of $1$. Adding $k-i$ to each side in (\ref{diseq}) gives $x_1-(i-1)> x_i$. Assume $\underline{x_1}$ is on the left of $x_k$. Since $x_2\neq x_1-1$, then  $x_i=x_1-1$ for some $i>2$ and $x_1-1>x_1-(i-1)> x_i$, a contradiction. Therefore, we may assume $\underline{x_1}$ is in on the right of $x_k$. Since $\overline {x}_{k-1}$ is on the right of $1$, the $1$-move
$$0|x_1\cdots x_k|\underline{x_k}\cdots\underline{x_1}|$$ turns $\bar\pi$ into $[0\,\underline{x_k}\cdots x_{k-1}\,x_k\cdots1\cdots\overline {x}_{k-1}\cdots]$. Now, there exists a $2$-move, namely $$\underline{x_k}|\cdots x_{k-1}|x_k\cdots1\cdots|\overline {x}_{k-1}.$$ Therefore, Proposition \ref{propEE} in case 3 follows. This concludes the proof of Proposition \ref{propEE}.

\section{Proof of Eriksson's bound}\label{tiger}

Let $\pi$ be a permutation on $[n]$ with $n\geq 4$. We apply Corollary \ref{kitty} after dismissing the case where $\pi$ is the reverse permutation by virtue of \cite[Theorem 4.3]{EE}. Assume that the first case occurs in Corollary \ref{kitty}, the other two cases may be investigated in the same way. By Proposition \ref{distance} and Corollary \ref{prop d},
\begin{equation}\label{fi}
d(\rho)\leq d(\rho\circ\sigma\circ\tau)+d(\tau^{-1})+d(\sigma^{-1}),
\end{equation} As the distance of a block transposition is $1$, the right-hand side in (\ref{fi}) is equal to $d(\rho\circ\sigma\circ\tau)+2$.
Collapsing bonds into a single symbol has the effect of collapsing $\rho\circ\sigma\circ\tau$ into a permutation on $[{n-3}]$. Then $d(\rho\circ\sigma\circ\tau)+2\leq d(n-3)+2$. By Theorem \ref{torically}, we obtain $d(\pi)=d(\rho)=\leq d(n-3)+2$, and then $$d(n)\leq d(n-3)+2.$$

Now the argument in the proof of \cite[Theorem 4.2]{EE} may be used to finish the proof of Eriksson's bound. This also shows that Eriksson's bound holds only by virtue of Theorem \ref{torically}.

\end{document}